\documentclass[11pt]{amsart}

\usepackage[all]{xy}
\usepackage{graphics}
\usepackage{color}
\usepackage[T1]{fontenc}
\usepackage{parskip}
\usepackage[colorlinks]{hyperref}
\usepackage{enumitem}
\usepackage[colorinlistoftodos,prependcaption]{todonotes}

\newcommand{\B}[1]{{\mathbf #1}}
\newcommand{\C}[1]{{\mathcal #1}}

\newtheorem{thm}{Theorem}[section]
\newtheorem{thm*}{Theorem}
\newtheorem{claim}[thm]{Claim}
\newtheorem{lem}[thm]{Lemma}
\newtheorem{prop}[thm]{Proposition}

\newtheorem{q*}[thm*]{Question}
\newtheorem{cor*}[thm*]{Corollary}

\theoremstyle{definition}

\newtheorem{defn*}[thm*]{Definition}
\newtheorem{ex}[thm]{Example}

\newtheorem{rem}[thm]{Remark}
\newtheorem{rem*}[thm*]{Remark}
\newtheorem{rems*}[thm*]{Remarks}

\newcommand{\OP}{\operatorname}

\begin{document}

\title[Fragmentation norm]{Fragmentation norm and
relative quasimorphisms}
\author{Michael Brandenbursky}
\address{Ben Gurion University of the Negev, Israel}
\email{brandens@math.bgu.ac.il}
\author{Jarek K\k{e}dra}
\address{University of Aberdeen and University of Szczecin}
\email{kedra@abdn.ac.uk}



\begin{abstract}
We prove that  manifolds with complicated enough fundamental group
admit measure-preserving homeomorphisms which have positive
stable fragmentation norm with respect to balls of bounded measure.
\end{abstract}

\maketitle 

\section{Introduction}\label{S:intro}

Homeomorphisms of a connected manifold $M$ can be often expressed as
compositions of homeomorphisms supported in sets of a given cover of~$M$.  This
is known as the {\em fragmentation property}.  Let $\OP{Homeo}_0(M,\mu)$ be the
identity component of the group of compactly supported measure-preserving
homeomorphisms of $M$. In this paper we are interested in groups $G\subseteq
\OP{Homeo}_0(M,\mu)$ consisting of homeomorphisms which satisfy the
fragmentation property with respect to topological balls of measure at most
one.  When $G$ is a subgroup of the identity component $\OP{Diff}_0(M, \mu)$ 
of the group of compactly supported volume-preserving diffeomorphisms,
we consider the fragmentation property with respect to smooth balls of volume
at most one.  Given $f\in G$, its {\em fragmentation norm} $\|f\|_{\text{\sc
frag}}$ is defined to be the smallest $n$ such that $f=g_1\cdots g_n$ and each
$g_i$ is supported in a ball of measure at most one.  Thus the fragmentation
norm is the word norm on $G$ associated with the generating set consisting of
maps supported in balls as above. We are also interested in the stable
fragmentation norm defined by $ \displaystyle{\lim_{k\to \infty}
\frac{\|f^k\|}{k}.}$ The existence of an element with positive stable
fragmentation norm implies that the diameter of the fragmentation norm is
infinite.  We say that the fragmentation norm on $G$ is {\em stably unbounded}
if $G$ has an element with positive stable fragmentation norm.  In general, a
group $G$ is called stably unbounded if it admits a stably unbounded
conjugation invariant norm; see Section \ref{S:preliminaries} for details.

The main result of the paper (Theorem \ref{T:general})  provides conditions
under which the existence of an essential homogeneous quasimorphism (i.e., a
quasimorphism which is not a homomorphism) on the fundamental group of $M$
implies the existence of an element of $G$ with a positive stable fragmentation
norm. Then we specify this abstract result to the following concrete cases.

{\bf Standing assumption.} Throughout the paper we consider manifolds $M$
such that the evaluation map
$\OP{ev}_x\colon\OP{Homeo}_0(M,\mu)\to M$, given by $\OP{ev}_x(f) = f(x)$ induces
the trivial homomorphism on fundamental groups. 
The only exception is Section \ref{SSS:closed} where we consider general
closed symplectic manifolds. The assumption is satisfied if, for example,
$M$ is either non-compact and connected or if the center of the fundamental
group $\pi_1(M)$ is trivial.  

\begin{thm}[Homeomorphisms]\label{T:homeo}
Let $M$ be a complete Riemannian manifold.
Let $\mu$ be the measure whose value on any smooth bounded set equals to its volume. 
Let $G$ be the kernel of the flux homomorphism $\OP{Homeo}_0(M,\mu)\to H_1(M,\B R)$. If $\pi_1(M)$ 
admits an essential quasimorphism then the fragmentation norm
on $G$ is stably unbounded.
\end{thm}

\begin{rem}
Let $\left[\OP{Homeo}_0(M,\mu),\OP{Homeo}_0(M,\mu)\right]$ be the
commutator subgroup of the identity component of the group of compactly
supported measure-preserving homeomorphisms of $M$. 
If $\dim M\geq 3$, then the commutator subgroup 
$\left[\OP{Homeo}_0(M,\mu),\OP{Homeo}_0(M,\mu)\right]$ is equal to the
kernel of flux, see \cite[Main Theorem]{MR584082}. If $\dim M=2$, then there is only
an inclusion and the equality is an open problem to the best of our knowledge.
For more information about the flux homomorphism see \cite[Section 3]{MR1445290}. 
\end{rem}

\begin{thm}[Diffeomorphisms]\label{T:diff}
Let $M$ be a complete Riemannian manifold equipped with a volume form $\mu$.
Let $G$ be the commutator subgroup of the identity component $\OP{Diff}_0(M,\mu)$ 
of the group of compactly supported volume-preserving diffeomorphisms of $M$. If
$\pi_1(M)$ admits an essential quasimorphism then the
fragmentation norm on $G$ is stably unbounded.
\end{thm}

\begin{thm}\label{T:ham}
Let $(M,\omega)$ be a symplectic manifold (satisfying the standing assumption as above). 
If $\pi_1(M)$ admits an essential quasi-morphism
then the fragmentation norm on the group $\OP{Ham}(M,\omega)$ of compactly supported 
Hamiltonian diffeomorphisms of $M$ is stably unbounded.
\end{thm}

\begin{rem}
Theorem \ref{T:ham} holds true when $(M,\omega)$ is either a general open
symplectic manifold or a closed symplectic manifold satisfying the standing
assumption. The statement also remains true for a general closed symplectic
manifold with a slightly modified fragmentation norm. We discuss the details
in Section \ref{SSS:closed}.
\end{rem}

\subsection*{Examples}
The following are simple examples of manifolds for which the stable unboundedness
of the fragmentation norm is a new result.

\begin{ex}\label{E:R3-L}
Let $M= \B R^{3} \setminus (L_1\cup L_2)$, where $L_i$ are disjoint lines.
Then the commutator subgroup of $\OP{Diff}_0(M,\mu)$ has stably unbounded fragmentation
norm, and in particular is stably unbounded. 
\hfill $\diamondsuit$
\end{ex}

\begin{ex}\label{E:R4-L}
Let $M=\B R^{4}\setminus (P_1\cup P_2)$, where $P_i$ are disjoint planes.
Suppose that $M$ is equipped with the standard symplectic form $\omega$ induced
from $\B R^4$. Then $\OP{Ham}(M,\omega)$ has stably unbounded fragmentation norm. 
\hfill $\diamondsuit$
\end{ex}

\begin{ex}\label{E:KB}
Let $M$ be the Klein Bottle with a point removed equip\-ped with
a measure from Theorem \ref{T:homeo}. Then for $k>0$ the group of
compactly supported measure-preserving homeomorphisms generated
by maps supported in balls of area at most $k$ has
stably unbounded fragmentation norm.
\hfill $\diamondsuit$
\end{ex}

\begin{rem}
Lanzat \cite{MR3142258} and Monzner-Vichery-Zapolsky \cite{MR2968955} showed
that the fragmentation norm on groups of Hamiltonian diffeomorphisms of certain
non-compact symplectic manifolds is stably unbounded.  Stable unboundedness of
the Hofer norm on the group $\OP{Ham}(\B R^2\setminus\{x,y\})$ was proved by
Polterovich-Sieburg \cite{MR1764319}. Stable unboundedness of
the fragmentation norm on the group $\OP{Ham}(\B R^2\setminus\{x,y\})$ follows from the results from
Monzner-Vichery-Zapolsky \cite{MR2968955}, as well as from our Theorem \ref{T:ham}.
For $M$ compact Theorem \ref{T:diff} was proven by Polterovich, see
\cite[Section 3.7]{MR2276956}.  
\end{rem}

\begin{rem}
Unboundedness of a group is usually proven with the use of unbounded
quasimorphisms. When $M$ is not of finite measure it is not known whether its groups of
transformations admit unbounded quasimorphisms.  Our proofs involve a
construction of relative quasimorphisms on groups in question.  We would like to
note that several constructions of relative quasimorphisms in symplectic
geometry appeared in \cite{MR2208798, MR2968955, kawasaki}.
\end{rem}

\subsection*{Acknowledgments}
We would like to thank Lev Buhovsky, Yaron Ostrover and Leonid Polterovich for
helpful discussions.  We would like to thank the referees for a thorough reading of our paper
and for very useful suggestions. 

We thank the Center for Advanced Studies in Mathematics
at Ben Gurion University for supporting the visit of the second author at BGU.
This work was funded by a Leverhulme Trust Research Project Grant
RPG-2017-159.

\section{Preliminaries}\label{S:preliminaries}

\subsection{Definitions}\label{SS:definitions}
Let $G$ be a group. A function  $\nu\colon G\to [0,\infty)$ is called a {\em
conjugation-invariant norm} if it satisfies the following conditions for all
$g,h\in G$:
\begin{enumerate}
\item
$\nu(g)=0$ if and only if $g=1_G$
\item
$\nu(g)=\nu(g^{-1})$
\item
$\nu(gh)\leq \nu(g) + \nu(h)$
\item
$\nu(ghg^{-1})=\nu(h)$.
\end{enumerate}
Let $\psi\colon G\to \B R$ be a function.
The {\em stabilization}
of $\psi$ is a function $\overline{\psi}\colon G\to \B
R$ defined by
$$
\overline{\psi}(g)=\lim_{n\to \infty}\frac{\psi\left( g^n \right)}{n},
$$
provided that the above limit exists for all $g\in G$. 
Note that for a norm $\nu$ its stabilization always exists 
because $\nu$ is a non negative subadditive function.

A norm $\nu$ is called {\em stably unbounded} if there exists $g\in G$ with
positive stabilization: $\overline{\nu}(g)>0$.  A group $G$ is called (stably)
{\em unbounded} if it admits a (stably) unbounded conjugation-invariant norm.
For more information about these notions see \cite{MR2509711}.  

A function $\psi\colon G\to\B R$ is called a {\em quasimorphism}
if there exists a real number $C\geq 0$ such that
$$
|\psi(gh) - \psi(g) - \psi(h)|\leq C
$$
for all $g,h\in G$. The infimum of such $C$'s is called the
\emph{defect} of $\psi$ and is denoted by $D_\psi$.
A quasimorphism $\psi$ is called {\em homogeneous} if
$$
\psi(g^k)=k\psi(g)
$$ 
for all $k\in \B Z$ and all $g\in G$. 
The stabilization of a quasimorphism exists and is a homogeneous quasimorphism \cite{MR2527432}.
A homogeneous quasimorphism is called {\em essential} if it is not a homomorphism.

A function $\psi: G\to \B R$ is called a {\em relative quasimorphism} with
respect to a conjugation-invariant norm $\nu$ if there exists a positive
constant $C$ such that for all $g,h\in G$
$$
|\psi(gh)-\psi(g)-\psi(h)|\leq C\min\{\nu(g),\nu(h)\}.
$$ 

For more information about quasimorphisms and their connections to different
branches of mathematics, see \cite{MR2527432}.

\subsection{The setup and assumptions}\label{SS:setup}
\hfill

{\bf (1)}
Let $M$ be a smooth complete Riemannian manifold.
Let $\C B$ be the set of all subsets
$B\subset M$ which are homeomorphic to the $n$-dimensional Euclidean unit 
open ball 
$$B^n=\{(x_1,\ldots,x_n)\in \B R^n\ |\ x_1^2+\cdots+x_n^2<1\}$$
and of measure at most one:
$$
\C B=\{B\subset M\ |\ B\cong B^n\ \text{and } \mu(B)\leq 1\}.
$$
Notice that the group of all measure-preserving homeomorphisms of $M$
acts on the set $\C B$. Let $\C B'$ be the set of all subsets
$B\subset M$ which are diffeomorphic to the $n$-dimensional Euclidean unit 
open ball $B^n\subset \B R^n$ and of volume at most one. 
Analogously, the group of all volume-preserving diffeomorphisms of $M$
acts on the set $\C B'$.

{\bf (2)} Let $G$ be a subgroup of $\OP{Homeo}_0(M,\mu)$. 
We assume that $G$ has the {\em fragmentation property} with respect
to the family $\C B$ (respectively the family $\C B'$ if it is a subgroup
of $\OP{Diff}_0(M,\mu)$). This means that for every $f\in G$ there exist
$g_1,\cdots,g_n\in G$ such that
\begin{enumerate}
\item $f=g_1\cdots g_n$ and
\item ${\rm supp}(g_i)\subset B_{g_i}\in \C B$ (respectively $\C B'$).
\end{enumerate}

{\bf (3)}
The {\em fragmentation norm} on $G$ associated with $\C B$ (respectively $\C
B'$) is defined by 
$$
\|f\|_{\text{\sc frag}} = \min\{n\in \B N\ |\ f = g_1\cdots g_n\},
$$
where $g_i$'s are as in the previous item. Notice that it is
a conjugation invariant norm.

{\bf (4)} 
Let $z\in M$ be a base-point. For each $x\in M$ let $\gamma_x\colon [0,1]\to M$
be a shortest geodesic path from $z$ to $x$. Such a geodesic is unique for all
$x$ away from the cut-locus (which is a set of measure zero \cite[Lemma
III.4.4]{MR1390760}) and on this set we also have that $\gamma_x$ depend
continuously on $x$.  Moreover, the paths $\gamma_x$ are of bounded lengths on
compact subsets, i.e., for each compact subset $K\subset M$ there exists a
constant $C_K>0$ such that $\ell(\gamma_x)\leq C_K$ for every $x\in K$. Here
$\ell(\gamma)$ denotes the Riemannian length of a path $\gamma$.  

{\bf (5)}
Let $\mu$ be a Lebesgue measure on $M$, i.e., a measure whose value on any
smooth bounded set is equal to its volume, and let $G\subseteq
\OP{Homeo}_0(M,\mu)$ be a subgroup
of the identity component of the group of compactly-supported mea\-sure-preserving
homeomorphisms of $M$.

{\bf (6)}
Let ${\rm ev_z}\colon G\to M$ be the {\em evaluation} map defined by
${\rm ev_z}(f)=f(z)$.  We say that $G$ has {\em trivial evaluation} if the homomorphism 
$${\rm ev_z}_*\colon \pi_1(G,1)\to \pi_1(M,z)$$
is trivial. For example, this is the case if $\pi_1(M,z)$ has trivial center because the 
image of ${\rm ev_z}_*$ lies in the center of $\pi_1(M,z)$, see e.g. \cite{MR32:6454}.
Observe that the above triviality condition is independent of the choice of the
basepoint. It follows that
if $M$ is non-compact then the above condition is always satisfied. 

{\bf (7)}
Let ${\psi}\colon \pi_1(M,z)\to \B R$ be a nontrivial homogeneous quasimorphism.
Let $f\in G$, and let $\{f_t\}$ be an isotopy from the identity to $f$. 
We take $f_t\in \OP{Homeo}_0(M,\mu)$.
Now we define $\Psi\colon G\to \B R$ by
$$
\Psi(f) = \int_M {\psi}([f_x]) \mu,
$$
where $f_x$ is a loop represented by the  concatenation 
$\gamma_x\cdot \{f_t(x)\}\cdot \overline{\gamma_{f(x)}}$.  
Note that if $\OP{Homeo}_0(M,\mu)$ 
has trivial evaluation then $[f_x]$ does not depend on the isotopy $\{f_t\}$.
Hence in this case $[f_x]$ also does not depend on the isotopy $\{f_t\}$.
Thus in order to show that $\Psi(f)$ is well-defined it is enough to show that $\Psi(f)<\infty$ is for each $f$.
Note that this construction appeared before in case $M$ is compact and $G=\OP{Diff}_0(M,\mu)$ in \cite{MR2276956}.

\begin{prop}\label{P:integrable}
The function $\Psi\colon G\to\B R$ is well-defined.
\end{prop}
\begin{proof}
Let $f\in G$. It is enough to show that the set $\{[f_x]\}_{x\in M}$ is finite.
Let $\{f_t\}\in\OP{Homeo}_0(M,\mu)$ be an isotopy  from the identity to $f$. 
The union of the supports $\bigcup_{t\in[0,1]}\OP{supp}(f_t)$ is a 
compact subset of $M$. Recall that $M$ admits a complete Riemannian metric.
Hence there exists $r>0$ such that the geodesic ball $B_r(z)$ of 
radius $r$ centered at $z$ contains $\bigcup_{t\in[0,1]}\OP{supp}(f_t)$.
Note that for each $x\in M\setminus B_r(z)$ the element $[f_x]$ 
is trivial in $\pi_1(M, z)$. Hence it is enough to show that the
set $\{[f_x]\}_{x\in B_r(z)}$ is finite in $\pi_1(B_r(z), z)$.

The ball $B_r(z)$ is compact, so we cover it with finite number of balls $B_i$, where each $B_i\in\C B$.
Then $f$ can be written as a product of measure-preserving homeomorphisms $h_i$
such that the support of $h_i$ lies in $B_i$. Since $M$ is a smooth manifold, for each $i$ there 
exits a smooth ball $B'_i$, such that it is $\epsilon$-close to $B_i$ and such that 
it is $\epsilon$-homotopic to $B_i$, see smooth approximation theorem \cite[Theorem 2.11.8]{MR1224675}.
Note that $[f_x]$ satisfies a cocycle condition. It means that
$$[f_x]=[(h_1\circ\ldots\circ h_n)_x]=[(h_1)_{(h_2\circ\ldots\circ h_n)(x)}]\ldots [(h_n)_x].$$
Let $r'$ be such that $B_{r'}(z)$ contains all balls $B_i$, $B'_i$ and the the images of $\epsilon$-homotopies
between them. Hence it is enough to prove that the set $\{([(h_i)_x]\}_{x\in B_i}$ is finite in $\pi_1(B_{r'}(z), z)$.

The ball $B'_i$ is smooth, thus it has finite diameter $d_i$. 
The group of measure-preserving homeomorphisms of a ball is connected, 
see \cite[Proposition 3.8]{MR584082}.
Every path inside $B_i$ can be free $\epsilon$-homotoped  
to a path in $B'_i$ and hence to a path whose Riemannian length is less than the diameter $d_i$. 
Thus $[(h_i)_x]$ can be represented by a path whose Riemannian length 
is less than $d_i+2(r_i+\epsilon)$, where $r_i$ is a radius of a geodesic ball $B_{r_i}(z)$ which contains $B_i$. 
By the Milnor-Schwarz lemma \cite{MR1744486} the word length of $[(h_i)_x]$ is bounded in $\pi_1(B_{r'}(z), z)$ and we are done.
\end{proof}

\subsection{The main technical result}
Before stating the theorem we summarize the assumptions we need.

{\bf Assumptions:}
\begin{itemize}[leftmargin=*]
\item $M$ is an $n$-dimensional complete Riemannian manifold.
\item $\mu$ is a Lebesgue measure on $M$.
\item $\C B$ is the set of topological balls in $M$ of measure
at most $1$.
\item $\C B'$ is the set of smooth balls in $M$ of measure
at most $1$.
\item $\OP{Homeo}_0(M,\mu)$ admits a trivial evaluation.
\item $G\subseteq\OP{Homeo}_0(M,\mu)$, 
$G$ is not a subgroup of $\OP{Diff}_0(M)$, 
$G$ has fragmentation property with respect to $\C B$.
\item $G\subseteq\OP{Diff}_0(M,\mu)$,  
$G$ has fragmentation property with respect to $\C B'$.
\end{itemize}

\begin{thm}\label{T:general}
Let $M$, $\mu$, $\C B$, $\C B'$ and $G$ be as above.
If there exists a homogeneous quasimorphism
$\psi\colon \pi_1(M)\to~\B R$ such that the homogenization $\overline{\Psi}$ is
nonzero then the fragmentation norm on $G$ is stably unbounded.
\end{thm}

In order to apply Theorem \ref{T:general} to a concrete group $G$ we need to
verify that the function $\overline{\Psi}\colon G\to \B R$ is nonzero.  It is
done with the use of various {\em push} maps. These applications are presented
in Section \ref{S:applications}.

\section{Proof of Theorem \ref{T:general}}
\begin{lem}\label{L:psi}
The function $\Psi\colon G\to\B R$ is a relative quasimorphism with
respect to the fragmentation norm.  
\end{lem}
\begin{proof}
Let $f,g\in G$ and let $\{f_t\}$ and $\{g_t\}$ be isotopies from the identity
to $f=f_1$, and to $g=g_1$ respectively. 
Recall that we can take both
isotopies $f_t$ and $g_t$ in $\OP{Homeo}_0(M,\mu)$ and not in $G$. 
Denote by $\{f_t\}*\{g_t\}$ the
concatenation of $\{g_t\}$ and $\{f_t\circ g\}$, so that it is an isotopy from
the identity to $fg$.  In what follows these isotopies are used to represent
$[f_{g(x)}]$, $[g_x]$ and $[(fg)_x]$ respectively.  Let us discuss several
cases:
\begin{enumerate}[leftmargin=*]
\item 
If $x\notin \bigcup_{t\in [0,1]}\OP{supp}(g_t)$, then $[(fg)_x]=[f_{g(x)}]$ and
$[g_x]$ is the identity element of $\pi_1( M, z)$.  It follows that for all such
$x$ we have
$${\psi}([(fg)_x])-{\psi}([f_{g(x)}])-{\psi}([g_x])=0.$$
\item 
If $x\notin \bigcup_{t\in [0,1]}g^{-1}(\OP{supp}(f_t))$, then $[(fg)_x]=[g_x]$
and $[f_{g(x)}]$ is the identity element of $\pi_1( M, z)$.  It follows that for
all such $x$ we have
$${\psi}([(fg)_x])-{\psi}([f_{g(x)}])-{\psi}([g_x])=0.$$
\end{enumerate}
Thus for every 
$$x\notin \left(\bigcup_{t\in [0,1]}
\OP{supp}(g_t)\right)\bigcap\left(\bigcup_{t\in [0,1]}g^{-1}(\OP{supp}(f_t))\right)$$
we have that
$$
{\psi}([(fg)_x])-{\psi}([f_{g(x)}])-{\psi}([g_x])=0.
$$

Recall that by result of Fathi the group 
$\OP{Homeo}(B, \mu)$ is connected where $B$ is a topological ball.
Let $\|g\|_{\text{\sc frag}}=k$ and $\|f\|_{\text{\sc frag}}=m$. It follows that 
$g=g_1\circ\ldots\circ g_k$ where each $g_i$ is supported in a ball $B_{i,g}$, and
$f=f_1\circ\ldots\circ f_m$ where each $f_i$ is supported in a ball $B_{i,f}$. Let 
$g_{i,t}$ and $f_{i,t}$ be the isotopies from the identity
to $g_i$ and $f_i$ respectively, such that they are
supported in $B_{i,g}$ and $B_{i,f}$ respectively. 
These isotopies lie in $\OP{Homeo}_0(M,\mu)$.
We choose $g_t$ to be a concatenation of isotopies $g_{i,t}$, 
and $f_t$ to be a concatenation of isotopies $f_{i,t}$ respectively. Hence 
\begin{align*}
\mu\Big (\bigcup_{t\in [0,1]}\OP{supp}(g_t)\Big )
\leq\|g\|_{\text{\sc frag}}\ \ \text{and}\
\quad\mu\Big(\bigcup_{t\in [0,1]}g^{-1}(\OP{supp}(f_t))\Big)
\leq\|f\|_{\text{\sc frag}}.
\end{align*} 
Let $U=\left(\bigcup_{t\in [0,1]}
\OP{supp}(g_t)\right)\bigcap\left(\bigcup_{t\in
[0,1]}g^{-1}(\OP{supp}(f_t))\right)$.
The two inequalities above imply that 
$$\mu(U)\leq \min\{\|f\|_{\text{\sc frag}},\|g\|_{\text{\sc frag}}\}$$
and we obtain that 
\begin{align*}
\big|\Psi(fg)-\Psi(f)-\Psi(g)\big|
&\leq \int_{ M}\Big|{\psi}([(fg)_x])-{\psi}([f_{g(x)}])
-{\psi}([g_x])\Big|\mu\\
&\leq \int_{U}D_{{\psi}}\ \mu\\
&\leq D_{{\psi}}\min\{\|f\|_{\text{\sc frag}},\|g\|_{\text{\sc frag}}\},
\end{align*}
where the first inequality follows from the fact that $[(fg)_x]=[f_{g(x)}][g_x]$ 
and $D_{{\psi}}$ is the defect of the quasimorphism ${\psi}$.
This shows that $\Psi$ is a relative quasi-morphism with respect
to the fragmentation norm.
\end{proof}

\begin{lem}\label{L:homogenization}
Let $\Psi\colon G\to \B R$ be the function defined in Section \ref{SS:setup}~(4).
Its homogenization $\overline{\Psi}\colon G\to \B R$ is well defined
and invariant under conjugations.
\end{lem}
\begin{proof}
If $M$ is a closed manifold then Polterovich proved that
$\overline{\Psi}$ is a homogeneous quasi-morphism
\cite[Section 3.7]{MR2276956}. In particular, it is
invariant under conjugation.
Since $\overline{\Psi}$ is evaluated on a compactly supported
homeomorphism, there exists a finite measure subset of $M$ containing
this support and the base-point and an argument of Polterovich implies
the statement.
\end{proof}

\begin{lem}\label{L:lipschitz}
The homogenization $\overline{\Psi}$ is Lipschitz with
respect to the fragmentation norm. More precisely,
$$
\big| \overline{\Psi}(f)\big|\leq 3\ D_{{\psi}}\ \|f\|_{\text{\sc frag}},
$$
for every $f\in G$.
\end{lem}
\begin{proof}
Let $f\in G$ be such that $\|f\|_{\OP{Frag}}=n$. This means that
there exist $g_1,\ldots,g_n\in G$ such that
$f=g_1\cdots g_n$ and each $g_i$ is supported in a 
ball of measure at most one. Then for $n>1$ we have
\begin{align*}
\big |\Psi(f)-\Psi(g_1)-\ldots-\Psi(g_n)\big|
&\leq \sum_{i=1}^{n-1}\Big|\Psi(g_1\cdots g_{i+1})-\Psi(g_1\cdots
g_i)-\Psi(g_{i+1})\Big|\\
&\leq\sum_{i=1}^{n-1} D_{{\psi}}<n D_{{\psi}},
\end{align*}
where the second inequality comes from the chain of inequalities at the end of
the proof of Lemma \ref{L:psi}. It follows from the adaptation of the proof
of Lemma 2.21 in \cite{MR2527432} to our case that for each $g\in G$ one has 
$$
|\overline{\Psi}(g)-\Psi(g)|
\leq D_{{\psi}}\ \mu(\OP{supp}(g))
\leq D_{{\psi}}\|g\|_{\OP{Frag}}.
$$
This leads to the following inequalities
\begin{align*}
&\quad \ |\overline{\Psi}(f)-\overline{\Psi}(g_1)-\ldots-\overline{\Psi}(g_n)|\\
&\leq |\overline{\Psi}(f)-\Psi(f)|+\sum_{i=1}^n|\overline{\Psi}(g_i)-\Psi(g_i)|+
|\Psi(f)-\Psi(g_1)-\ldots-\Psi(g_n)|\\
&\leq D_{{\psi}}\|f\|_{\OP{Frag}}+\sum_{i=1}^n D_{{\psi}}+nD_{{\psi}}=3nD_{{\psi}}.
\end{align*}
Recall that each $g_i$ is supported in a ball. 
By result of Fathi  \cite[Proposition 3.8]{MR584082} the group 
$\OP{Homeo}(B, \mu)$ is connected where $B$ is a topological ball.
It follows that there exists a measure preserving isotopy $g_{t,i}$ from
the identity to $g_i$ which is supported in a ball.
If this ball is smooth, then
for each $x\in M$ element $[(g_i^p)_x]\in\pi_1( M, z)$ can be represented
by a path whose Riemannian length is bounded by some constant independent of $p$,
see proof of Proposition \ref{P:integrable}.
If this ball is not smooth, then there exists a smooth 
ball which is $\epsilon$-homotopic to this ball, see \cite[Theorem 2.11.8]{MR1224675},
and hence the same statement holds by an argument from the proof of Proposition \ref{P:integrable}.
It follows from the Milnor-Schwarz lemma that for each $i$
\begin{equation}\label{Eq:valuation-0}
\overline{\Psi}(g_i)=
\int\limits_{ M} \lim_{p\to\infty}\frac{{\psi}([(g_i^p)_x])}{p}\mu=0.
\end{equation}
Combining \eqref{Eq:valuation-0} with previous inequalities we get
\begin{equation}\label{Eq:frag-inequality}
|\overline{\Psi}(f)|\leq 3nD_{{\psi}}
=3D_{{\psi}}\|f\|_{\text{\sc frag}}
\end{equation}
which shows that $\overline{\Psi}$ is Lipschitz with respect to the
fragmentation norm.
\end{proof}

\begin{proof}[Proof of Theorem \ref{T:general}]
Recall that the hypothesis says that there exists a non-trivial
homogeneous quasimorphism ${\psi}\colon \pi_1( M)\to\B R$ 
such that the relative quasimorphism $\overline{\Psi}$ is nontrivial. 
Take $f\in G$ such that $\overline{\Psi}(f)\neq 0$. 

Now \eqref{Eq:frag-inequality} yields
$$
\lim_{k\to\infty}\frac{\|f^k\|_{\OP{Frag}}}{k}
\geq\lim_{k\to\infty}\frac{|\overline{\Psi}(f^k)|}{3kD_{{\psi}}}
=\lim_{k\to\infty}\frac{k|\overline{\Psi}(f)|}{3kD_{{\psi}}}
=\frac{|\overline{\Psi}(f)|}{3D_{{\psi}}}>0,
$$
which proves that the fragmentation norm is stably unbounded.
\end{proof}

\section{Applications of Theorem \ref{T:general}}\label{S:applications}

\subsection{The commutator subgroup of the group of homeomorphisms}
\label{SS:homeo}
Consider $\B S^1\times \B D^{n-1}$, where $\B D^{n-1}$ is the closed
$(n-1)$-dimensional Euclidean disc of radius $1+\epsilon$, where $\epsilon >0$
is an arbitrarily small number.  Let $\varphi_s\colon \B D^{n-1}\to \B R$, for
$s\in [0,1]$ be a family of smooth functions supported away from the boundary
such that it is equal to $s$ on the disc of radius~$1$.  Let
$f_s\colon \B S^1\times \B D^{n-1}\to \B S^1\times  \B D^{n-1}$ be defined by
$f_s(t,z)=(t+\varphi_s(z),z)$; here $\B S^1 = \B R/\B Z$.
It is straightforward to verify that $f_s$ preserves
the standard product Lebesgue measure. Let $\gamma\colon \B S^1\to M$ be an
embedded loop and let $\widehat{\gamma}\colon \B S^1\times \B D^{n-1}\to M$ be
a measure preserving embedding such that $\widehat{\gamma}(t,0)=\gamma(t)$.  
Here we need to assume that with respect to the product measure the disk
$\B D^{n-1}$ has sufficiently small measure.
We define the
associated push-map $f_{\widehat{\gamma}}\colon M\to M$ by
$$
f_{\widehat{\gamma}}(x) = 
\begin{cases}
\widehat{\gamma}\circ f_1\circ\widehat{\gamma}^{-1}(x) & \text{for } x\in
\widehat{\gamma}\left( \B S^1\times \B D^{n-1} \right) \\
\OP{Id} & \text{otherwise}.
\end{cases}
$$
Changing the parameter $s$ defines an isotopy from the identity to $f_{\widehat{\gamma}}$
through measure preserving homeomorphisms. More precisely, the isotopy is defined by
$$
f_{\widehat{\gamma},s}(x) =
\begin{cases}
\widehat{\gamma}\circ f_s\circ\widehat{\gamma}^{-1}(x) & \text{for } x\in
\widehat{\gamma}\left( \B S^1\times \B D^{n-1} \right) \\
\OP{Id} & \text{otherwise}.
\end{cases}
$$
A similar construction appears in Fathi \cite{MR584082}.

Let $\B D_1\subset\B D^{n-1}$ denote the $n-1$ dimensional disc of radius $1$ centered at zero.
It follows that for each $x\in\widehat{\gamma}\left( \B S^1\times \B D_1 \right)$ we get
$[(f_{\widehat{\gamma}}^p)_x]=\gamma^p$,
where we view $\gamma$ as an element of $\pi_1(M,z)$. For 
each $x\in\widehat{\gamma}\left( \B S^1\times (\B D^{n-1}\setminus\B D_1) \right)$ we get
$$[(f_{\widehat{\gamma}}^p)_x]=\alpha_{f,x,p}\cdot\gamma^{p_{f,x}}\cdot\alpha'_{f,x,p},$$
where $|p_{f,x}|<|p|$, and the word length of $\alpha_{f,x,p}$ and $\alpha'_{f,x,p}$ is 
bounded by a constant which is independent of $p$, see e.g. \cite[Section 2.D.1]{MR3488375}.
This yields
$$\left|\overline{\Psi}(f_{\widehat{\gamma}})-\psi(\gamma)\mu(\OP{supp}(f_{\widehat{\gamma}}))\right|\leq
|\psi(\gamma)|\mu\left(\B S^1\times (\B D^{n-1}\setminus\B D_1)\right)$$
It follows that we may choose the constant $\epsilon$
suitably small so that the value $\overline{\Psi}(f_{\widehat{\gamma}})$ is arbitrarily
close to $\psi(\gamma)\mu(\OP{supp}(f_{\widehat{\gamma}}))$. 

Let $G$ be the kernel of the flux homomorphism. It has the fragmentation property
by \cite[Theorem A.6.5]{MR584082}. 
Theorem \ref{T:homeo} is a consequence of the following result.

\begin{prop}\label{P:stably-unb}
If $\psi\colon \pi_1(M)\to \B R$ is an essential quasimorphism then the
fragmentation norm on $G$ is stably unbounded.
\end{prop}
\begin{proof}
For simplicity we denote elements of the fundamental group $\pi_1(M,z)$
and their representing loops by the same Greek letters.
Since $\psi$ is an essential quasimorphism, there exist
$\alpha,\beta\in \pi_1(M)$ such that 
$$
\big|\psi(\alpha) -\psi(\alpha\beta)+\psi(\beta)\big|=a>0.
$$
Let $\alpha,\beta\colon \B S^1\to M$ be embedded based loops representing the
above elements of the fundamental group. If $\dim M\geq 3$ then such loops
exists for all elements of the fundamental group for dimensional reasons.  If
$\dim M=2$ then if $\pi_1(M)$ admits an essential quasimorphism then $\pi_1(M)$ is
either non-abelian free or the surface group of higher genus. 
In this case $\pi_1(M)$ has abundance of essential
quasimorphisms \cite{MR1452851} and we can choose $\psi$ and embedded loops $\alpha,\beta$ which
satisfy the above requirement.

Consider the push maps
$f_{\widehat{\alpha}},f_{\widehat{\beta}}$ such that their support have equal measures and
$\overline{\Psi}(f_{\widehat{\alpha}})$ is arbitrarily close to
$\psi(\alpha)\mu(\OP{supp}(f_{\widehat{\alpha}}))$ and similarly for~$f_{\widehat{\beta}}$. 

\begin{claim}\label{L:push-homogeneous}
With the above notation we have that
$$
|\overline{\Psi}(f_{\widehat{\alpha}})
-\overline{\Psi}(f_{\widehat{\alpha}}f_{\widehat{\beta}})
+\overline{\Psi}(f_{\widehat{\beta}})| 
\geq a \mu(\OP{supp}(f_{\widehat{\alpha}})\cap\OP{supp}(f_{\widehat{\beta}}))-\delta >0,
$$
where $\delta>0$ is a constant which can be made arbitrarily small
by a suitable choice of the push maps. 
\end{claim}
The proof of the above claim is straightforward and relies on the observation
that the subsets of the supports of the push maps where they vary between the
identity and the rotation can be made arbitrarily small. 
It is similar to the proof of Lemma \ref{L:psi} and similar arguments
are presented in \cite{MR3391653, MR3181631}.

It follows that the relative quasimorphism
$\overline{\Psi}\colon \OP{Homeo}_0(M,\mu)\to \B R$ 
is nonzero and it is not a homomorphism.
Hence it must be nontrivial on the commutator subgroup $\left[\OP{Homeo}_0(M,\mu),\OP{Homeo}_0(M,\mu)\right]$
because a relative quasimorphism is a genuine quasimorphism when restricted
to homeomorphisms supported on any fixed subset of finite measure. Nontriviality
of a quasimorphism which is not a homomorphism on the commutator subgroup is straightforward.
Since $\left[\OP{Homeo}_0(M,\mu),\OP{Homeo}_0(M,\mu)\right]<G$, 
the statement follows from a direct application of Theorem \ref{T:general}.
\end{proof}

\begin{rem}
Let $\dim M=2$. Then $\left[\OP{Homeo}_0(M,\mu),\OP{Homeo}_0(M,\mu)\right]<G$, and 
it is not known whether it admits fragmentation property. However, it admits an 
induced fragmentation norm from $G$. The proof of Proposition \ref{P:stably-unb} 
shows that this norm is stably unbounded. 
\end{rem}

\subsection{The commutator subgroup of $\OP{Diff}_0(M, \mu)$} \label{SS:diff}
The fragmentation property is due to Thurston; see Banyaga \cite[Lemma
5.1.2]{MR1445290}.  The proof of Theorem \ref{T:diff} is analogous to
the above proof for homeomorphisms in the sense that the push map is obtained
from the same maps 
$$f_s\colon \B S^1\times \B D^{n-1}\to \B S^1\times \B D^{n-1}$$ 
which are transplanted to $M$ via differentiable maps. Then the
application of Proposition \ref{P:stably-unb} is the same.

\subsection{The group of Hamiltonian diffeomorphisms}\label{SS:ham}
The fragmentation property is due to Banyaga \cite[page 110]{MR1445290}.
Here the strategy is the same but the construction of the
push map needs more care. This is done as follows.

\begin{proof}[Proof of Theorem \ref{T:ham}]
Let $T=\B S^1\times [-1,1]\times \B D^{2n-2}$ be equipped with the product of an
area form on the annulus and the standard symplectic form on the disc.  The
coordinate on $\B S^1$ is denoted by $x$, on $[-1,1]$ by $y$ and a point in the
disc by $z$. So the symplectic form is $dx\wedge dy + \omega_0$.  
Let $\varphi\colon \B D^{2n-2}\to \B R$ be a non-negative function supported in the
interior of the disc and equal to $1$ on a disc of radius arbitrarily close to the
radius of $\B D^{2n-2}$. Let $f\colon [-1,1]\to \B R$ be a smooth function supported
in the interior of the interval $[-1,1]$.
Let $H\colon T\to \B R$ be defined by 
$H(x,y,z)=yf(y)\varphi(z)$.
Then 
$$dH = \varphi (f+yf') dy + yfd\varphi$$
and the corresponding Hamiltonian vector field is given by
$$
X_H(x,y,z) = \varphi(z)(f(y)+yf'(y)) \partial_x + Z(y,z),
$$
where $Z$ is a suitable vector field on the disc which depends on $y$. If 
$$
f(y)=
\begin{cases}
e^{-\frac{1}{1-y^2}} & y\in (-1,1)\\
0 &\text{ otherwise }
\end{cases}
$$ 
is the standard bell-shaped function then the equation 
$f(y) + yf'(y)=0$ has two solutions $\pm y_0 = \pm \sqrt{2-\sqrt{3}}$ 
in the interval $[-1,1]$ and we have that $f(y)+yf'(y)>0$
for $y\in (-y_0,y_0)$. We restrict the vector field $X_H$ to the subset $\B
S^1\times (-y_0,y_0)\times \B D^{2n-2}$ and extend it by zero to the rest of
$T$. Notice that this vector field is symplectic but not Hamiltonian and
it points in the non-negative direction of $\partial_x$.

The rest of the proof is the same as in the case of homeomorphisms.  That is, we
choose the classes $\alpha,\beta\in \pi_1(M)$ for which 
$$\psi(\alpha\beta)\neq \psi(\alpha)+\psi(\beta).$$ 
Then we choose their embedded
representatives and choose their small tubular neighborhoods symplectically
diffeomorphic to $T$ with possibly rescaled summands of the symplectic form
(\cite[Exercise 3.37]{MR2000g:53098}. We transplant the above symplectic flows
to create $f_{\widehat{\alpha}},f_{\widehat{\beta}}\in \OP{Symp}_0(M,\omega)$ and the same argument
shows that $\overline{\Psi}$ is nontrivial and that it is not a homomorphism and
hence it is nontrivial on the commutator subgroup of $\OP{Symp}_0(M,\omega)$ 
which is the group $\OP{Ham}(M,\omega)$ of Hamiltonian diffeomorphisms.  
\end{proof}

\subsubsection{The case of a closed symplectic manifold
$(M,\omega)$}\label{SSS:closed} Recall that if the induced evaluation map on
$\pi_1(\OP{Homeo}_0(M,\mu), 1)$ is trivial then Theorem \ref{T:ham} holds. Let
us discuss the general case without assuming the above triviality.

It follows from the proof of Arnold's conjecture that the induced evaluation
map on $\pi_1(\OP{Ham}(M,\omega),1)$ is trivial, see \cite[Exercise
11.28]{MR2000g:53098}; see also discussion in \cite{arXiv:1209.4410} about the
proof of Arnold's conjecture in full generality.  Thus we can modify the
construction of the map $$\Psi\colon \OP{Ham}(M,\omega)\to \B R$$ defined in
{\bf (7)} by taking isotopies only in $\OP{Ham}(M,\omega)$. Note that this map
is a quasimorphism, and was defined by Polterovich in \cite{MR2276956}. Its
homogenization $\overline{\Psi}$ is an essential quasimorphism whenever
$\psi\colon\pi_1(M,z)\to\B R$ is an essential quasimorphism. Thus in order to
show that the fragmentation norm is stably unbounded it is enough to show that
an essential quasimorphism
$$
\overline{\Psi}\colon \OP{Ham}(M,\omega)\to \B R
$$ 
vanishes on Hamiltonian diffeomorphisms supported in balls of volume at most
one,  see e.g. \cite[Lemma 5.2]{BM-entropy}. 

Notice that the proof of~\eqref{Eq:valuation-0} in Lemma \ref{L:lipschitz} does
not work in this case because it is not known whether a Hamiltonian
diffeomorphism of $M$ supported in a ball $B\in \C B'$ can be isotoped to the
identity through a Hamiltonian isotopy supported in $B$ (such an issue does
not occur in dimension two \cite[Corollary (2.6)]{MR1487633}).  We overcome
this problem by adjusting the definition of the fragmentation norm as follows.
We define it by
$$
\|f\|_{\text{\sc frag}} = \min\{n\in \B N\ |\ f = g_1\cdots g_n\},
$$
where ${\rm supp}(g_i)\subset B_{g_i}\in \C B'$ and $g_i\in\OP{Ham}(B_{g_i})$.
Note that in this case the fragmentation norm is well defined, since 
$\OP{Ham}(M,\omega)$ is a simple group when $M$ is closed \cite{MR1445290}.
After this modification the proof of~\eqref{Eq:valuation-0} 
in Lemma~\ref{L:lipschitz} goes through as before and so does the proof
of the version of Theorem~\ref{T:ham} for the fragmentation norm defined
as above.

\subsection{An example for a question of Burago, Ivanov and Polterovich}
One of the open problems of  \cite{MR2509711} asks for an example of a group
which is perfect, has stably vanishing commutator length and admits a stably
positive conjugation invariant norm. The following example, has been suggested
by an anonymous referee.

Consider the group $H=\OP{Ham}(T^*\Sigma_g\times \B R^{2n})$, where $\Sigma_g$ is
a closed surface of genus $g>1$. The commutator subgroup of $H$ is perfect and
has stably vanishing commutator length (every compact subset of $T^*\Sigma_g\times \B R^{2n}$
 is displaceable by a compactly supported Hamiltonian isotopy
which lies in the commutator subgroup). Observe that the fragmentation norm on
$H$ restricted to the commutator subgroup $G=[H,H]$ is still stably unbounded
because the relative quasimorphism $\overline{\Psi}$ does not vanish on the
commutator subgroup. So this restriction yields a required example for
the Burago, Ivanov, Polterovich problem.

\bibliography{bibliography}
\bibliographystyle{plain}

\end{document}